\documentclass[11pt]{amsart}
\usepackage{amsfonts}
\usepackage{}
\usepackage{mathrsfs}
\usepackage{bbm}
\usepackage{amssymb}
\usepackage{indentfirst}
\usepackage{latexsym,amsfonts,amssymb,amsmath}
\newtheorem{theorem}{Theorem}[section]
\newtheorem{lemma}[theorem]{Lemma}
\newtheorem{corollary}[theorem]{Corollary}
\newtheorem{question}[theorem]{Question}
\theoremstyle{definition}
\newtheorem{definition}[theorem]{Definition}
\newtheorem{proposition}[theorem]{Proposition}
\theoremstyle{remark}
\newtheorem{remark}[theorem]{Remark}
\newtheorem{example}[theorem]{Example}

\begin{document}

\title
{on $\star$-metric spaces}

\author{Shi-yao He }\thanks{}
\address{(S.Y. He) School of Mathematics and Computational Science, Wuyi University, Jiangmen 529020, P.R. China} \email{2644481035@qq.com}
\author{Li-Hong Xie}\thanks{}
\address{(L.H. Xie) School of Mathematics and Computational Science, Wuyi University, Jiangmen 529020, P.R. China} \email{yunli198282@126.com; xielihong2011@aliyun.com}
\author{Peng-Fei Yan}
\address{(P.-F. Yan)School of Mathematics and Computational Science, Wuyi University, Jiangmen 529020, P.R. China}\email{ypengfei@sina.com}

\thanks{This research is supported by NSFC (No.11861018), the Natural Science Foundation of Guangdong
Province under Grant (No.2021A1515010381).}

\begin{abstract}
Metric spaces are generalized by many scholars. Recently, Khatami and Mirzavaziri use a mapping called $t$-definer to popularize the triangle inequality and give a generalization of the notion of a metric, which is called a $\star$-metric. In this paper, we prove that every $\star$-metric space is metrizable. Also, we study the total boundedness and completeness of $\star$-metric spaces.
\end{abstract}

\maketitle

\section{Introduction}
It is well known that metric spaces are widely used in analysis. There are several common metric spaces, such as the numerical straight line $\mathbb{R}$, the $n$-dimensional Euclidean space $\mathbb{R}^n$, the continuous function space and the Hilbert space. Therefore, every metric space is an important kind of topological space.  A function $d:X\times X \rightarrow \mathbb{R}^+$ is called a {\it metric} on a set $X$ if $d$ satisfies the following conditions for every $x, y,z\in X$:

(1) $d(x,y)=0$ if and only if $x=y$;

(2) $d(x,y)=d(y,x)$;

(3) $d(x,y)\leqslant d(x,z) + d(z,y)$.

Then set $X$ with $d$ is called a {\it metric space}, denoted by $(X,d)$.

We can obtain the generalizations of metric spaces when we weaken or modify the conditions of metric axiom. Pseudo-metrics are obtained when we change that condition `$\rho(x,y)=0$ if and only if $x=y$' into `$\rho(x,y)=0$ if $x=y$' \cite{GGS}. Quasi-metrics are defined by omitting the condition (2) \cite{14}. Symmetrics are defined by omitting the triangle inequality \cite{dc}. The ultrametric is a metric with the strong triangle inequality $d(x, y) \leqslant \max\{d(x, z), d(z, y)\}$, for $x, y,z\in X$ \cite{11}. There are many generalizations of metric spaces which have appeared in literatures (e.g. see \cite{6,4,3,9,5,CR})

Recently, Khatami and Mirzavaziri popularized the concept of metric. By extending the famous function which is called $t$-conorm, a new operation called $t$-definer is obtained. It is defined as:

\begin{definition} \label{def1} \cite[Definition 2.1]{AG} A {\it $t$-definer} is a function $\star$ : $[0,\infty)\times[0,\infty) \rightarrow [0,\infty)$ satisfying the following conditions for each $a, b,c \in [0,\infty)$:

\begin{enumerate}
\item[(T1)] $a \star b =b \star a$;
\item[(T2)] $a \star (b\star c) =(a \star b)\star c$;
\item[(T3)] if $a \leqslant b$, then $a \star c \leqslant b\star c $;
\item[(T4)] $a \star 0 = a$;
\item[(T5)] $\star$ is continuous in its first component with respect to the Euclidean topology.
\end{enumerate}
\end{definition}
When the condition (3) in the metric axiom is extended to the $\star$-triangle inequality, the following definition of $\star$-metrics can be obtained.

\begin{definition}\label{def2} \cite[Definition 2.2]{AG} Let $X$ be a nonempty set and $\star$ is a $t$-definer. If for every $x, y, z\in X$, a function $d^\star:X\times X \rightarrow [0,\infty)$ satisfies the following conditions:
\begin{enumerate}
\item[(M1)] $d^\star(x,y)=0$ if and only if $x=y$;
\item[(M2)] $d^\star(x,y)=d^\star(y,x)$;
\item[(M3)] $d^\star(x, y)\leqslant d^\star(x,z)\star d^\star(z,y)$,
\end{enumerate}
then $d^\star$ is called a {\it $\star$-metric} on $X$. The set $X$ with a $\star$-metric is called {\it $\star$-metric space}, denoted by $(X,d^\star)$.
\end{definition}

The following example shows that there are $\star$-metrics which are not metrics.

\begin{example}\cite[Example 2.4.]{AG} Clearly, $a \star b =(\sqrt{a}+\sqrt{b})^2$ is a t-definer. The function $d^\star(a,b) =(\sqrt{a}-\sqrt{b})^2$ forms an $\star$-metric which is not a metric. In fact,

\begin{equation}
\begin{aligned}
d^\star(a,b) &=(\sqrt{a}-\sqrt{b})^2=(\sqrt{a}-\sqrt{c}+ \sqrt{c}-\sqrt{b})^2\\
&\leqslant[\sqrt{(\sqrt{a}-\sqrt{c})^2}+\sqrt{(\sqrt{c}-\sqrt{b})^2}]^2\\
&=[\sqrt{d^\star(a,c)}+\sqrt{d^\star(c,b)}]^2\\
&=d^\star(a,c)\star d^\star(c,b)\nonumber
\end{aligned}
\end{equation}
while $d^\star(1,25)=16\nleqslant d^\star(1,16)+d^\star(16,25)=10.$
\end{example}

\begin{remark}
There are two most important $t$-definers:

$\bullet$ Lukasiewicz $t$-definer: $a \star_{L} b =a+b$;

$\bullet$ Maximum $t$-definer: $a \star_{m} b =\max{\{a,b\}}.$

Obviously, a $\star_L$-metric is actually a metric and a $\star_m$-metric is an ultrametric.
\end{remark}



Assume that $(M, d^\star)$ is a $\star$-metric space. For any $a\in M$ and $r > 0$, denote by
$$B_{d^\star}(a, r)=\{x\in M: d^\star(a,x)<r\}$$ and $$\mathscr{T}_{d^\star}=\{U\subseteq M: \text{~for each~}a\in U\text{~there is~}r>0 \text{~such that~}B_{d^\star}(a,r)\subseteq U\}.$$

Let $\mathscr{B}_{\frac{1}{n}}=\{B_{d^\star}(x,\frac{1}{n})\mid x\in X\}$ be a family open balls on a $\star$-metric space $(X, d^\star)$.

Khatami and Mirzavaziri proved the following result:

\begin{theorem}\cite[Theorems 3.2, 3.4, 3.5]{AG}
For every $\star$-metric space $(X, d^\star)$, the $\mathscr{T}_{d^\star}$ forms a Hausdorff topology on $X$ and the topological space $(X, \mathscr{T}_{d^\star})$ is first countable and satisfied the normal separation axiom.
\end{theorem}

The following question is posed naturally.

\begin{question}\label{q1}
Is the topological space $(X, \mathscr{T}_{d^\star})$ metrizable for every $\star$-metric space $(X, d^\star)$?
\end{question}

In this paper, we give a positive answer to Question \ref{q1}. Also, we extend the concepts of the totally boundedness and completeness in metric spaces into $\star$-metric spaces and discuss their properties.

This paper is organized as follows. Section 2 is given a positive answer to Question \ref{q1}. We obtain that let $(X, d^\star)$ be a $\star$-metric space; then the set $X$ with the topology $\mathscr{T}_{d^\star}$ induced by $d^\star$ is metrizable (see Theorem \ref{Th1}). In Section 3 total boundedness of $\star$-metric spaces are studied. We prove that: (1) let $(X,d^\star)$ be a totally bounded $\star$-metric space; then for every subset $M$ of $X$ the $\star$-metric space $(M,d^\star)$ is totally bounded (see Theorem \ref{thm2}); (2) let $(X,d^\star)$ be a $\star$-metric space and for every subset $M$ of $X$ the space  $(M,d^\star)$ is totally bounded if and only if $(\overline{M},d^\star)$ is totally bounded (see Theorem \ref{thm2.2}). We show that the Cartesian product and disjoint union of finite totally bounded $\star$-metric spaces are totally bounded under specific $\star$-metrics (see Theorems \ref{thm4} and \ref{thm5.2}). In Section 4, the completeness of $\star$-metric spaces are explored. We obtain that: (1) A $\star$-metric space $(X,d^\star)$ is compact if and only if $(X,d^\star)$ is complete and totally bounded (see Theorem \ref{thm13});
(2) A $\star$-metric space is complete if and only if for every decreasing sequence $F_{1}\supseteq F_{2}\supseteq F_{3}\supseteq\ldots$ of non-empty closed subsets of space $X$, such that $\lim_{n\rightarrow \infty }\delta(F_{n})=0$, the intersection $\bigcap_{n=1}^\infty F_{n}$ is a one-point set (see Theorem \ref{thm7});
(3) the completeness of the Cartesian product and disjoint union of complete $\star$-metric spaces under specific $\star$-metrics (see Theorems \ref{thm11} and \ref{thm11.2}). 
(4) In a complete $\star$-metric space $(X,d^\star)$ the intersection $A=\bigcap_{n=1}^\infty A_{n}$ of a sequence $A_{1},A_{2},\dots$ of dense open subsets is a dense set (see Theorem \ref{thm12}).

\section{Metrizability of $\star$-metric spaces}
In this section, we shall give a positive answer to Question \ref{q1}. Let $(X,d^\star)$ be a $\star$-metric space, we shall prove that the set $X$ with the topology $\mathscr{T}_{d^\star}$ induced by $d^\star$ is metrizable. We need to use some related symbols, terms, and preliminary facts.

Let $\mathscr{U}$ be a cover of a set $X$. For $x\in X$, denote by st($x,\mathscr{U}$)=$\bigcup\{U:U \in \mathscr{U},x\in U\}$ and st($A,\mathscr{U}$)=$\bigcup_{x\in A}\text{st}(x,\mathscr{U})$ for $A\subseteq X$. Let $\mathscr{U}$, $\mathscr{V}$ be two covers of a set $X$, we say that the cover $\mathscr{V}$  {\it refines} $\mathscr{U}$, if for every $V\in \mathscr{V}$, there exists an $U\in \mathscr{U}$ such that $V\subset U$, denoted by $\mathscr{U}< \mathscr{V}$; if $\{st(V,\mathscr{V}):V \in \mathscr{V}\} $ refines $\mathscr{U}$, then we said that $\mathscr{V}$ {\it star refines} $\mathscr{U}$, denoted by $\mathscr{U} \mathop{<}\limits_{}^*\mathscr{V}$.

\begin{definition} \cite[Definition 4.5.1]{GGS}\label{def} Let $X$ be a set and $\Phi=\{\mathscr{U}_\alpha:\alpha \in A\}$ a non-empty collection of covers of $X$ which satisfies:
\begin{enumerate}
\item[(U1)]  if $\mathscr{U}$ is a cover of $X$ such that $\mathscr{U}_\alpha < \mathscr{U}$ for some $\alpha\in A$, then $\mathscr{U}\in\Phi$;

\item[(U2)] for any $\alpha, \beta\in A$, there exists an $\gamma\in A$ such that $\mathscr{U}_\gamma < \mathscr{U}_\alpha$, $\mathscr{U}_\gamma < \mathscr{U}_\beta$;

\item[(U3)] for every $\alpha\in A$, there exists an $\beta\in A$ such that $\mathscr{U}_\beta \mathop{<}\limits_{}^* \mathscr{U}_\alpha$;

\item[(U4)] if $x, y$ are different elements of $X$, then $x\notin \text{st}(y,\mathscr{U}_\alpha)$ for some $\alpha\in A$.
\end{enumerate}

Then $X$ is called a {\it uniform space} with the uniformity $\Phi$, denoted by $(X,\Phi)$. Let $X$ be a uniform space with the uniformity $\Phi=\{\mathscr{U}_\alpha:\alpha \in A\}$ and let $\{\mathscr{U}_\beta:\beta \in B\}$ be a subcollection of $\Phi$. If for every $\alpha\in A$, there exists a $\beta\in B$ such that $\mathscr{U}_\beta < \mathscr{U}_\alpha$, then the collection $\Phi'$ is called {\it a basis of the uniformity}.

Let $X$ be a uniform space with the uniformity $\Phi=\{\mathscr{U}_\alpha:\alpha \in A\}$
and
$$\mathscr{T}_{\Phi}=\{U:U\subseteq X, \text{~for each~}x\in U, \text{~there is~} \alpha \in A \text{~such that~} \text{st}(x,\mathscr{U}_{\alpha})\subset U\}.$$ Then $\mathscr{T}_{\Phi}$ is a topology on the $X$.
\end{definition}


Recalled that a topological space $X$ is said to be {\it metrizable} if there exists a metric $d$ on the set $X$ that induces the topology of $X$.




\begin{lemma}\cite[Theorem 4.5.9]{GGS}\label{lem2} Let $(X,\Phi)$ be a uniform space. Then the set with the topology $\mathscr{T}_{\Phi}$ induced by $\Phi$ is metrizable if and only if there is a base of the uniformity consisting of countably many covers.
\end{lemma}


%




 The following theorem shows that $\star$ is continuous at the point $(0,0)$.

\begin{lemma}\label{lem3}
For $r>0$, there exists $r_{1}>0$ such that $[0,r_{1})\star[0,r_{1})\subseteq [0,r)$.
\end{lemma}
\begin{proof}
For $r>0$, we have $0 \star \frac{1}{2} r \in [0,r)$ by  [Definition \ref{def1}, (T4)]. According to [Definition \ref{def1}, (T5)], there exists an $r_{0}>0$ such that $[0,r_{0})\star \frac{1}{2} r\subseteq [0,r)$. Without loss of generality, let $0\leqslant r_{0} \leqslant r$ by [Definition\ref{def1},(T3)]. Take $r_{1}=\frac{1}{2} r_{0}$. Then, we can claim that $[0,r_{1})\star[0,r_{1})\subseteq [0,r)$. For every $x, y\in [0,r_{1})$, we have that $$ x\star y\leqslant x\star \frac{1}{2} r.$$

Noting that $x\star \frac{1}{2} r \in [0,r_{1})\star \frac{1}{2} r \subseteq [0,r_{0}) \star \frac{1}{2} r \subseteq [0,r)$, we have that $x\star \frac{1}{2} r \in [0,r)$, which means $x\star \frac{1}{2} r < r$. Since $ x\star y\leqslant x\star \frac{1}{2} r $, $ x\star y < r $, therefore $x, y\in [0,r)$.
\end{proof}

The following theorem gives a positive answer to Question \ref{q1}.

\begin{theorem} \label{Th1}
Let $(X, d^\star)$ be a $\star$-metric space. Then the set $X$ with the topology $\mathscr{T}_{d^\star}$ induced by $d^\star$ is metrizable.
\end{theorem}

\begin{proof}
First we shall show that $\mathscr{B}=\{\mathscr{B}_\frac{1}{n}:n \in \mathbb{N}\}$ is a base of a uniformity on set $X$, where $\mathscr{B}_\frac{1}{n}=\{B_{d^\star}(x,\frac{1}{n})\mid x\in X\}$. Indeed, it is enough to show that $\mathscr{B}$ satisfies (U2)$\sim$(U4) in Definition\ref{def1}.

(U2). For any $x \in X$, for any $n_{1},n_{2} \in \mathbb{N}$, take $n_{0} \in \mathbb{N}$, such that $n_{0}>n_{1}$, $n_{0}>n_{2}$. Take any $y\in B_{d^\star}(x,\frac{1}{n_{0}})$, then we have that $d^\star(x,y)< \frac{1}{n_{0}}< \frac{1}{n_{1}}$, thus $y\in B_{d^\star}(x,\frac{1}{n_{1}})$. This implies that $B_{d^\star}(x,\frac{1}{n_{0}})\subset B_{d^\star}(x,\frac{1}{n_{1}})$. Therefore $\mathscr{B}_\frac{1}{n_{0}}$ $< $ $\mathscr{B}_\frac{1}{n_{1}}$.
Similarly, we can prove that $\mathscr{B}_\frac{1}{n_{0}}< \mathscr{B}_\frac{1}{n_{2}}$. So, $\mathscr{B}$ satisfies (U2).

(U3). For any $n_{0}\in \mathbb{N}$, by Lemma \ref{lem3}, there exists an $r_{1}\in \mathbb{N}$ such that $r_{1}\star r_{1}\star r_{1}< \frac{1}{n_{0}}$. Take $n_{1}\in \mathbb{N}$ such that $\frac{1}{n_{1}}< r_{1}$. Now we shall prove that $\mathscr{B}_\frac{1}{n_{1}} \mathop{<}\limits_{}^* \mathscr{B}_\frac{1}{n_{0}}$. Hence the proof is completed once we show that st($B_{d^\star}(x,\frac{1}{n_{1}})$, $\mathscr{B}_\frac{1}{n_{1}}$)$\subseteq B_{d^\star}(x,\frac{1}{n_{0}})$, for any $x\in X$. Take any $y\in X$ such that $B_{d^\star}(y,\frac{1}{n_{1}}) \cap B_{d^\star}(x,\frac{1}{n_{1}})\neq \emptyset$. Then, there exists an $z_{1}\in B_{d^\star}(y,\frac{1}{n_{1}}) \cap B_{d^\star}(x,\frac{1}{n_{1}})$, for any $z_{2}\in B_{d^\star}(y,\frac{1}{n_{1}})$, we have
\begin{equation}
\begin{aligned}
d^\star(z_{2},x) &\leqslant d^\star(z_{2},y)\star d^\star(y,z_{1})\star d^\star(z_{1},x)\\
&<\frac{1}{n_{1}}\star \frac{1}{n_{1}}\star \frac{1}{n_{1}}<r_{1}\star r_{1}\star r_{1}\\
&< \frac{1}{n_{0}}.\nonumber
\end{aligned}
\end{equation}

Therefore, $z_{2}\in B_{d^\star}(y,\frac{1}{n_{0}})$, $\mathscr{B}$ satisfies (U3).

(U4). For $x,y \in X$ and $x\neq y $, let $d^\star(x,y)=r$ , then $r>0$. Take $n_{1} \in \mathbb{N}$, such that $\frac{1}{n_{1}}< r$, then we have $y \notin B_{d^\star}(x,\frac{1}{n_{1}})$. By Lemma \ref{lem3}, there exists an $n_{0} \in \mathbb{N}$ such that $\frac{1}{n_{0}}\star\frac{1}{n_{0}}< \frac{1}{n_{1}}$. Then, we claim that for any $B_{d^\star}(z_{0},\frac{1}{n_{0}})\in \mathscr{B}_\frac{1}{n_{0}}$ $(z_{0} \in X)$, $B_{d^\star}(z_{0},\frac{1}{n_{0}})$ can not contain both $x$ and $y$. Since, if $x,y\in B_{d^\star}(z_{0},\frac{1}{n_{0}})$, then $d^\star(x,y)\leqslant d^\star(x,z_{0}) \star d^\star(z_{0},y) < \frac{1}{n_{0}}\star\frac{1}{n_{0}} < \frac{1}{n_{1}} < r$, which is contradictory to $d^\star(x,y)=r$. So, $\mathscr{B}$ satisfies (U4).

 Thus we have proved that $\mathscr{B}$ is a base of a uniformity on $X$, denote by $\Phi_{d^\star}$. According to Lemma \ref{lem2}, the set $X$ with the topology $\mathscr{T}_{\Phi_{d^\star}}$ induced by the uniformity $\Phi_{d^\star}$ is metrizable. Therefore, to complete the proof, it is enough to prove that the topology $\mathscr{T}_{d^\star}$ induced by $d^\star$ is the same as the topology $\mathscr{T}_{\Phi_{d^\star}}$.

(1) For any $U\in \mathscr{T}_{d^\star}$, $x\in U$, there exists $n \in \mathbb{N}$  such that $B_{d^\star}(x,\frac{1}{n})\subset U$. By Lemma \ref{lem3}, there exists an $n_{0} \in \mathbb{N}$ such that $\frac{1}{n_{0}}\star\frac{1}{n_{0}}< \frac{1}{n}$. Then we shall prove that  $\text{st}(x,\mathscr{B}_{\frac{1}{n_{0}}}) \subset B_{d^\star}(x,\frac{1}{n})\subset U$. Take any $z\in \text{st}(x,\mathscr{B}_{\frac{1}{n_{0}}})$, then there exists $y \in X$  such that $x\in \text{st}(y,\mathscr{B}_{\frac{1}{n_{0}}})$ and $z\in \text{st}(y,\mathscr{B}_{\frac{1}{n_{0}}})$. Since,
\begin{equation}
\begin{aligned}
d^\star(x,z) &\leqslant d^\star(x,y)\star d^\star(y,z)<\frac{1}{n_{0}}\star\frac{1}{n_{0}}<\frac{1}{n},\nonumber
\end{aligned}
\end{equation}
we have $z\in \text{st}(x,\mathscr{B}_{\frac{1}{n}})$, i.e. $\text{st}(x,\mathscr{B}_{\frac{1}{n_{0}}}) \subset U$, so $U\in \mathscr{T}_{\Phi_{d^\star}}$. Therefore $ \mathscr{T}_{d^\star} \subseteq \mathscr{T}_{\Phi_{d^\star}}$.

(2) For any $U\in \mathscr{T}_{\Phi_{d^\star}}$, $x\in U$, there exists $n \in \mathbb{N}$  such that $\text{st}(x,\mathscr{B}_{\frac{1}{n}})\subset U$. Since  $B_{d^\star}(x,\frac{1}{n})\subset\text{st}(x,\mathscr{B}_{\frac{1}{n}}) \subset U$, we have $U\in \mathscr{T}_{d^\star}$, i.e. $ \mathscr{T}_{\Phi_{d^\star}}\subseteq \mathscr{T}_{d^\star}$. Then we get that $ \mathscr{T}_{\Phi_{d^\star}}= \mathscr{T}_{d^\star}$.
\end{proof}

Obviously, since $\star$-metric spaces are metrizable, we have the following two corollaries.

\begin{corollary} Let $(X,d^\star)$ be a $\star$-metric space and $X$ with the topology $\mathscr{T}_{d^\star}$ induced by $d^\star$. Then the following statements equivalent:
\begin{enumerate}
\item[(1)] X has a countable basis;
\item[(2)] X is Lindel\"{o}f;
\item[(3)] every closed discrete subspace is a countable set;
\item[(4)] every  discrete subspace is a countable set;
\item[(5)] every collection of disjoint open sets in X is  countable;
\item[(6)] X has a countable dense subset.

\end{enumerate}
\end{corollary}

Recalled that a topological space $X$ is called {\it pseudo-compact} if every real valued continuous function on $X$ is bounded; $X$ is called {\it countably compact} if every countable open cover of $X$ has a finite subcover; $X$ is called {\it sequentially compact} if every sequence of points of $X$ has a convergent subsequence. A point $x$ is called {\it $\omega$-accumulation point} of $X$ if any neighborhood of point $x$ contains infinite points of $X$.

\begin{corollary} \label{col1} Let $(X,d^\star)$ be a $\star$-metric space and $X$ with the topology $\mathscr{T}_{d^\star}$ induced by $d^\star$. Then the following statements equivalent:
\begin{enumerate}
\item[(1)] X is pseudo-compact;
\item[(2)] every infinite subset of X has cluster point;
\item[(3)] every infinite subset of X has $\omega$-accumulation point;
\item[(4)] every sequence of X has cluster point;
\item[(5)] X is countably compact;
\item[(6)] X is sequentially compact;
\item[(7)] X is compact.
\end{enumerate}
\end{corollary}

\section{Total boundedness of $\star$-metric spaces}
Total boundedness is an important property in metric spaces. We generalize the concept of totally bounded into $\star$-metric spaces and study their properties. Now we need to give some related definitions.

\begin{definition}
A $\star$-metric space $(X,d^\star)$ is {\it totally bounded}, if for any $\epsilon >0$ there exists a finite set $F(\epsilon)\subseteq X$ such that $X=\bigcup_{x\in F(\epsilon)}B_{d^\star}(x,\epsilon)$. We also said that the finite set $F(\epsilon)$ is
{\it $\epsilon$-dense} in $X$.
\end{definition}

\begin{theorem}\label{thm3} Let $(X,d^\star)$ be a $\star$-metric space and every infinite subset of $X$ have an $\omega$-accumulation point in the topological space $X$ with the topology induced by $d^ \star$. Then $(X,d^\star)$ is totally bounded.
\end{theorem}

\begin{proof}
Suppose the contrary that there exists $\epsilon_{0}$, but there is no finite set $F(\epsilon_{0})$, so that $X=\bigcup_{x\in F(\epsilon_{0})}B_{d^\star}(x,\epsilon_{0})$ holds. Take any $x_{1}\in X$, then $X\neq B_{d^\star}(x_{1},\epsilon_{0})$; take any $x_{2}\in X-B_{d^\star}(x_{1},\epsilon_{0})$, since $X\neq \bigcup_{i=1}^2 B_{d^\star}(x_{i},\epsilon_{0})$; Repeat this procedure, we obtain an infinite set $\{x_{1},x_{2},\dots ,x_{n},\dots\}$. According to the above operation, we can get that $d^\star(x_{i},x_{j})\geqslant \epsilon_{0}$ $(i\neq j)$. By assuming, this infinite set has an $\omega$-accumulation point $x_{0}\in X$. Thus, by Lemma \ref{lem3}, take an $\epsilon_{1}> 0$ such that $\epsilon_{1}\star \epsilon_{1}<\epsilon_{0}$. Then the open-ball $B_{d^\star}(x_{0},\epsilon_{1})$ should contain infinite points of $\{x_{1},x_{2},\dots ,x_{n},\dots\}$. Let $x_{n},x_{m}\in B_{d^\star}(x_{0},\epsilon_{1})$, then $$d^\star(x_{m}, x_{n}) \leqslant d^\star(x_{m}, x_{0})\star d^\star(x_{0}, x_{n})< \epsilon_{1}\star \epsilon_{1}< \epsilon_{0}.$$ This contradicts $d^\star(x_{i},y_{j})\geqslant \epsilon_{0}$. Therefore $(X,d^\star),$ is totally bounded.
\end{proof}

By Corollary \ref{col1} and Theorem \ref{thm3}, we have the following corollary:

\begin{corollary} \label{col2}
 Let $(X,d^\star)$ be a $\star$-metric space. If $X$ with the topology induced by $d^ \star$ is countably compact, then  $(X,d^\star)$ is a totally bounded $\star$-metric space.
\end{corollary}

One can easily verify that for every subset $M\subseteq X$ of a $\star$-metric space $(X,d^\star)$, $(M, d^\star\mid_{M\times M})$ is a $\star$-metric space, where $d^\star\mid_{M\times M}$ is the restriction of the $\star$-metric $d^\star$ on $X$ to the subset $M$. The $\star$-metric space $(M, d^\star\mid_{M\times M})$ will also be denoted by $(M,d^\star)$.


\begin{theorem}\label{thm2} Let $(X,d^\star)$ be a totally bounded $\star$-metric space. Then for every subset $M$ of $X$ the $\star$-metric space $(M,d^\star)$ is totally bounded.
\end{theorem}

\begin{proof} Let $(X,d^\star)$ be a $\star$-metric space, $M\subset X$. For every $\epsilon> 0$, by Lemma \ref{lem3}, take an $\epsilon_{1}> 0$ such that $\epsilon_{1}\star \epsilon_{1}<\epsilon$. Take a finite set $F(\epsilon_{1})=\{x_{1},x_{2},\dots ,x_{k}\}$  which is $\epsilon_{1}$-dense in $X$. Let $\{x_{m_{1}},x_{m_{2}},\dots ,x_{m_{l}}\}$ be the subset of $F(\epsilon_{1})$ consisting of all points which satisfy that $d^\star(x,x_{i})< \epsilon_{1}$, for each $x\in M$ and $x_{i}\in F(\epsilon_{1})$.
Let $F'=\{x'_{1},x'_{2},\dots ,x'_{l}\}$ be a finite set satisfying $d^\star(x'_{j},x_{m_{j}})< \epsilon_{1}$ for $j\leqslant l$. We shall show that the set $F'$ is $\epsilon$-dense in $M$. Let $x\in M$, there exists $x_{i}\in F(\epsilon_{1})$ such that $d^\star(x,x_{i})< \epsilon_{1}$. Hence $x_{i}=x_{m_{j}}$ for some $j\leqslant l$, then we have $$d^\star(x, x'_{j}) \leqslant d^\star(x, x_{m_{j}})\star d^\star(x_{m_{j}}, x'_{j})< \epsilon_{1}\star \epsilon_{1}< \epsilon.$$
\end{proof}

\begin{theorem}\label{thm2.2} Let $(X,d^\star)$ be a $\star$-metric space and for every subset $M$ of $X$ the space  $(M,d^\star)$ is totally bounded if and only if $(\overline{M},d^\star)$ is totally bounded.
\end{theorem}

\begin{proof}
Assume that $(M,d^\star)$ is totally bounded. For $\epsilon> 0$, Lemma \ref{lem3}, take an $\epsilon_{1}> 0$ such that $\epsilon_{1}\star \epsilon_{1}<\epsilon$, and take a finite set $F(\epsilon_{1})=\{x_{1},x_{2},\dots ,x_{k}\}$  which is $\epsilon_{1}$-dense in $M$. For each $x\in \overline{M}$ , we have $B_{d^\star}(x,\epsilon_{1})\cap M\neq \emptyset $. Take $y\in B_{d^\star}(x,\epsilon_{1})\cap M$ there exists $x_{i}\in F(\epsilon_{1})$ such that $d^\star(y,x_{i})< \epsilon_{1}$, then we have $d^\star(x, x_{i}) \leqslant d^\star(x, y)\star d^\star(y, x_{i})< \epsilon_{1}\star \epsilon_{1}< \epsilon.$

On the other hand, assume that $(\overline{M},d^\star)$ is totally bounded. One can easily get that $(M,d^\star)$ is totally bounded by Theorem \ref{thm2}, because $M$ is the subset of $\overline{M}$.
\end{proof}

\begin{corollary}
If the $\star$-metric space $(X,d^\star)$ has a dense totally bounded subspace, then $(X,d^\star)$ is totally bounded.
\end{corollary}

 Let $\{(X_{i},d_{i}^\star)\}_{i=1}^n $ be a family of finite nonempty $\star$-metric spaces. Consider the Cartesian product $X=\prod_{i=1}^{n}X_{i}$ and for every pair $x=(x_i)_{i\leq 1\leq n}$,  $y=(y_i)_{i\leq 1\leq n}$ of points of $X$ let
 $$d_{T}^\star(x,y)=d_{1}^\star(x_{1},y_{1})\star d_{2}^\star(x_{2},y_{2})\star \dots \star d_{n}^\star(x_{n},y_{n})\quad\quad\quad\quad\quad\quad(3.1)$$
 and
 $$d_{\max}^\star(x,y)=\max_{1\leqslant i\leqslant n} d_{i}^\star(x_{i},y_{i})\quad\quad\quad\quad\quad\quad\quad\quad\quad\quad\quad\quad\quad\quad\quad\quad(3.2).$$
In \cite[Theorem 4.3]{AG}, Khatami and Mirzavaziri proved that the formulas (3.1) and (3.2) define two $\star$-metrics on the Cartesian product $X=\prod_{i=1}^{n}X_{i}$. Furthermore the induced topology of these two $\star$-metrics on $X$ is the same as the product topology on $X$.

\begin{theorem}\label{thm4} Let $\{(X_{i},d_{i}^\star)\}_{i=1}^n $ be a family of finite nonempty $\star$-metric spaces and $X=\prod_{i=1}^{n}X_{i}$ the Cartesian product. Then:
\begin{enumerate}
\item[(1)] $X$ with the $\star$-metric $d_{T}^\star$ defined by formula (3.1) is totally bounded if and only if all $\star$-metric spaces $(X_{i},d_{i}^\star)$ are totally bounded;

\item[(2)] $X$ with the $\star$-metric $d_{max}^\star$ defined by formula (3.2) is totally bounded if and only if all $\star$-metric spaces $(X_{i},d_{i}^\star)$ are totally bounded.
\end{enumerate}
\end{theorem}

\begin{proof}
(1). Necessity. Assume that the $\star$-metric space $(X,d_{T}^\star)$ is totally bounded. The subset $X_{m}^*=\prod_{i=1}^{n}A_{i}$ of $X$ , where $A_{m}=X_{m}$ and $A_{i}=\{x_{i}^*\}$ is a one-point subset of $X_{i}$ for $i \neq m$. Then the subspace $X_{m}^*$ is  totally bounded by Theorem \ref{thm2}. One can easily verify that $p_{m}^*=p_{m}\mid _{X_{m}^*}:X_{m}^*\rightarrow X_{m}$ is a isometric isomorphism and according to the definition of $d_{T}^\star$, for $x^*,y^*\in X_{m}^*\subset X$, $d_{T}^\star(x^*,y^*)=d_{m}^\star(p_{m}(x^*),p_{m}(y^*))$. Therefore, if a finite set $F$ is $\epsilon$-dense in $(X_{m}^*,d_{T}^\star)$, then $p_{m}(F)$ is $\epsilon$-dense in $(X_{m},d_{T}^\star)$, and from this it follows further that $(X_{m},d_{T}^\star)$ is totally bounded.

Sufficiency. Let every $(X_{i},d_{i}^\star)$ is totally bounded. For $\epsilon> 0$, by Lemma \ref{lem3}, take an $\epsilon_{1}> 0$ such that $ \overbrace{\epsilon_{1}\star \epsilon_{1}\star \dots \star \epsilon_{1}  }^{n \text{ times}}<\epsilon$. For every $i\leqslant n$ take a finite set $F_{i}$ which is $\epsilon_{1}$-dense in $X_{i}$. We define that $$F=\prod_{i=1}^{n}F_{i},$$ then $F$ is a finite set. To conclude the proof it suffices to show that $F$ is $\epsilon$-dense in the space $(X,d_{T}^\star).$ Let $x=(x_{1},x_{2},\dots ,x_{n})$ be an arbitrary point of $X$. For every $i\leqslant n$, since $F_{i}$ is $\epsilon$-dense in $X_{i}$, there exists a $y_{i} \in F_{i}$ such that $d_{i}^\star(x_{i},y_{i})< \epsilon_{1}$ and take a point $y=(y_{1},y_{2},\dots ,y_{n}) \in F$ we have $$d_{T}^\star(x,y)=d_{1}^\star(x_{1},y_{1})\star d_{2}^\star(x_{2},y_{2})\star \dots \star d_{n}^\star(x_{n},y_{n})<\overbrace{\epsilon_{1}\star \epsilon_{1}\star \dots \star \epsilon_{1}  }^{n \text{ times}}<\epsilon.$$ By the foregoing, $F$ is $\epsilon$-dense in $(X,d_{T}^\star)$.

(2). Necessity. Assume that the $\star$-metric space $(X,d_{max}^\star)$ is totally bounded. Then the proof method is the same as that of necessity in (1).

Sufficiency. Let every $(X_{i},d_{i}^\star)$ is  totally bounded. For $\epsilon> 0$, take a finite set $F_{i}$  which is $\epsilon$-dense in $X_{i}$, for every $i\leqslant n$. We define that $$F=\prod_{i=1}^{n}F_{i},$$ then $F$ is a finite set. To conclude the proof it suffices to show that $F$ is $\epsilon$-dense in the space $(X,d_{\max}^\star)$. Let $x=(x_{1},x_{2},\dots ,x_{n})$ be an arbitrary point of $X$. For every $i\leqslant n$, there exists a $y_{i}\in F_{i}$ such that $d_{i}^\star(x_{i},y_{i})< \epsilon$ and  take a point $y=(y_{1},y_{2},\dots ,y_{n}) \in F$. Without loss of generality, choose $\max_{1\leqslant i\leqslant n} d_{i}^\star(x_{i},y_{i})=d_{k}^\star(x_{k},y_{k})$, then we have $$d_{\max}^\star(x,y)=\max_{1\leqslant i\leqslant n}d_{i}^\star(x_{i},y_{i})=d_{k}^\star(x_{k},y_{k})< \epsilon.$$ By the foregoing, $F$ is $\epsilon$-dense in $(X,d_{\max}^\star)$.

This completes the proof.
\end{proof}





Let $(X,d)$ be a metric space. Define $\tilde{d}(x,y)=\text{min}\{1,d(x,y)\}$ for each $x,y\in X$. It is well known that $\tilde{d}$ is a metric on $X$ such that the topology induced by $\tilde{d}$ is the same as induced by $d$. For $\star$-metric space, we have the following result.

\begin{proposition}\label{thm1} Let $(X,d^\star)$ be a $\star$-metric space and define $ \tilde{d^\star}(x,y)=\text{min}\{1,d^\star(x,y)\}$ for each $x,y\in X$. Then $(X,\tilde{d^\star})$ is also a $\star$-metric space on $X$. Furthermore the topology induced by $\tilde{d^{\star}}$ on $X$ is the same as induced by $d^{\star}$.
\end{proposition}
\begin{proof}
We shall verify that $\tilde{d^\star}$ is a $\star$-metric. Clearly, $\tilde{d^\star}$ satisfies (M1) and (M2) in the definition \ref{def2}. Suppose the contrary that exists $x,y,z\in X$ such that $$1\geqslant \tilde{d^\star} (x,z) > \tilde{d^\star} (x,y) \star \tilde{d^\star} (y,z),$$
then, $\tilde{d^\star} (x,y)<1$, $\tilde{d^\star} (y,z)<1$ (since, if $\tilde{d^\star}(x,y)\geqslant 1$, then $ \tilde{d^\star} (x,y)\star \tilde{d^\star} (y,z)\geqslant 1\star 0=1$, i.e $\tilde{d^\star} (x,z)>1$). Therefore $$\tilde{d^\star} (x,y)\star \tilde{d^\star}(y,z)=d^\star(x,y) \star d^\star(y,z)\geqslant d^\star(x,z).$$ This implies that $\tilde{d^\star} (x,z)>d^\star(x,z)$, which is a contradiction with $\tilde{d^\star} (x,z) \leqslant d^\star(x,z)$. Thus $(X,\tilde{d^\star})$ is a $\star$-metric space.

 For any $\epsilon >0$, $x\in X$ we define $$B_{d^\star}(x,\epsilon)=\{y \in X:d^\star(x,y)< \epsilon\},$$ $$B_{\tilde{d^\star}}(x,\epsilon)=\{y \in X:d^\star(x,y)< \epsilon\}.$$
When $0< \epsilon <1$, $B_{d^\star}(x,\epsilon)=B_{\tilde{d^\star}}(x,\epsilon)$. Thus the topology induced by $\tilde{d^{\star}}$ on $X$ is the same as induced by $d^{\star}$. This completes the proof.
\end{proof}


Let $\{(X_{\alpha},d^\star_{\alpha})\}_{\alpha \in A}$ be a family of $\star$-metric spaces and $X=\bigoplus_{\alpha\in A}X_{\alpha}$ be the disjoint union of $\{X_{\alpha}\}_{\alpha \in A}$. By Proposition \ref{thm1}, one can suppose that $d^\star_{\alpha}(x,y)\leqslant 1$ for $x,y\in X_{\alpha}$ and $\alpha \in A$. For every $x,y\in X$, we define
$$d^\star_{q}(x,y)=
\begin{cases}
d^\star_{\alpha}(x,y),& \text{if $x,y\in X_{\alpha}$ for some $\alpha \in A$},\\
1,& \text{otherwise}.\quad\quad\quad\quad\quad\quad\quad\quad\quad\quad\quad\quad\quad\quad\quad(3.3)
\end{cases}$$
Then $(X,d^\star_{q})$ is a $\star$-metric space.

Obviously, $d^\star_{q}$ satisfies conditions (M1) and (M2). It remains to show that condition (M3) $d^\star_{q}(x,z)\leqslant d^\star_{q}(x,y)\star d^\star_{q}(y,z)$ is also satisfied. Otherwise, if there exists $x,y,z\in X$, such that $1\geqslant d^\star_{q}(x,z)> d^\star_{q}(x,y)\star d^\star_{q}(y,z)$, then $d^\star_{q}(x,y)< 1$, $d^\star_{q}(y,z)<1$. Since, if $d^\star_{q}(x,y)\geqslant 1$, then $d^\star_{q}(x,y)\star d^\star_{q}(y,z)\geqslant 1\star 0=1$, i.e. $d^\star_{q}(x,z)>1$. This implies that $d^\star_{q} (x,z)>d^\star_{\alpha}(x,z)$, which is a contradiction with $d^\star_{q}(x,z) \leqslant d^\star_{\alpha}(x,z)$. Thus, there exists $\alpha \in A$ such that $x,y,z\in X_{\alpha}$, then we have $$d^\star_{q}(x,y)\star d^\star_{q}(y,z)=d^\star_{\alpha}(x,y)\star d^\star_{\alpha}(y,z)\geqslant d^\star_{\alpha}(x,z)=d^\star_{q}(x,z),$$ contradiction.

One can easily show that for every ${\alpha\in A}$, the set $X_{\alpha}$ is open in the space $X$ with the topology induced by $d^\star_{q}$. Since $d^\star_{\alpha}$ induces the topology on $X_{\alpha}$, then $d^\star_{q}$ induces on $X$ the topology of the disjoint union of topological spaces $\{X_{\alpha}\}_{\alpha\in A}$.

\begin{theorem}\label{thm5.2}
Let $\{(X_{i},d_{i}^\star)\}_{i=1}^n$ be a family of $\star$-metric spaces such that the metric $d_{i}^\star$ is bounded by $1$ for $1\leqslant i\leqslant n$, and $X=\bigoplus_{1\leqslant i\leqslant n} X_{i}$ the disjoint union of $\{X_{i}\}_{i\leq n}$. Then $(X,d^\star_{q})$ is totally bounded if and only if all spaces $(X_{i},d_{i}^\star)$ are totally bounded, where $d^\star_{q}$ is defined as the formula $(3.3)$.
\end{theorem}

\begin{proof}
Necessity. Assume that $(X,d^\star_{q})$ is totally bounded. One can easily get that $(X_{i},d_{i}^\star)$ is a subspace of $(X,d^\star_{q})$. So, all spaces $(X_{i},d_{i}^\star)$ are totally bounded by Theorem \ref{thm2}.

Sufficiency. Assume that all spaces $(X_{i},d_{i}^\star)$ are totally bounded. Then, for $\epsilon>0$, for any $x\in X_{i}$, there exists $y_{0}\in F_{i}(\epsilon)$ such that $d_{i}^\star(x,y_{0})<\epsilon$. Put $F(\epsilon)=\bigcup_{i=1}^n F_{i}(\epsilon)$, let $x$ be an arbitrary point of $X$, obviously, $x$ is also a point on some $X_{i}$. Thus, one can easily find a $y_{0}\in F_{i}(\epsilon)$, such that $d^\star_{q}(x,y_{0})=d_{i}^\star(x,y_{0})<\epsilon$. So, $(X,d^\star_{q})$ is totally bounded.
\end{proof}

\section{the completeness of $\star$-metric spaces}
Completeness is an important property in metric spaces. The completeness of metric spaces depend on the convergence of Cauchy sequences. Therefore, we extend the definition of Cauchy sequences and completeness in metric spaces to $\star$-metric spaces. Further, we study completeness properties of $\star$-metric spaces and give their characterization.

\begin{definition} Let $\{x_{n}\}_{n \in \mathbb{N}}$  be a sequence of a $\star$-metric space $(X,d^\star)$, and $x\in X$. If for every $\epsilon >0$, there exists ${k \in \mathbb{N}}$ such that $d^\star(x,x_{n})< \epsilon$ whenever $n\geqslant k$, then the sequence $\{x_{n}\}_{n \in \mathbb{N}}$ is said to {\it converge to $x$ under $d^\star$}, and we write $x_{n} \stackrel{d^\star}{\longrightarrow}x$.
\end{definition}

\begin{proposition} Let $(X,d^\star)$ be a $\star$-metric space. Then the following statements equivalent:
\begin{enumerate}
\item[(1)] $\{x_{n}\}_{n \in \mathbb{N}}$ converges to $x_{0}$ under  $ \mathscr{T}_{d^\star}$;
\item[(2)] $\{x_{n}\}_{n \in \mathbb{N}}$ converges to $x_{0}$ under  $ d^\star$.
\end{enumerate}
\end{proposition}

\begin{proof}
(1) $\Rightarrow$ (2). For every $\epsilon >0$, clearly, $B_{d^\star}(x_{0},\epsilon)$ is a neighborhood of $x_{0}$. Since $\{x_{n}\}_{n \in \mathbb{N}}$ converges to $x_{0}$ under  $ \mathscr{T}_{d^\star}$, there exists ${k \in \mathbb{N}}$ such that $x_{n} \in B_{d^\star}(x_{0},\epsilon)$ whenever $n\geqslant k$, i.e. $d^\star(x_{n},x_{0})< \epsilon$. Therefore $\{x_{n}\}_{n \in \mathbb{N}}$ converges to $x_{0}$ under  $ d^\star$.

(2) $\Rightarrow$ (1). For any neighborhood $U$ of the point $x_{0}$, there exists $\epsilon >0$ such that $B_{d^\star}(x_{0},\epsilon)\subset U$. Since $\{x_{n}\}_{n \in \mathbb{N}}$ converges to $x_{0}$ under  $ d^\star$, there exists $k \in \mathbb{N}$ such that $d^\star(x_{n},x_{0})< \epsilon$ whenever $n\geqslant k$, i.e. $x_{n} \in B_{d^\star}(x_{0},\epsilon)$. Thus $x_{n}\in U$ whenever $n\geqslant k$. Therefore $\{x_{n}\}_{n \in \mathbb{N}}$ converges to $x_{0}$ under $\mathscr{T}_{d^\star}$.
\end{proof}

\begin{definition}
Let $(X,d^\star)$ be a $\star$-metric space, the sequence $\{x_{n}\}$ is called {\it Cauchy sequence} in $(X,d^\star)$ if for every $\epsilon>0$ there exists $k \in \mathbb{N}$ such that $d^\star(x_{n},x_{m})< \epsilon$ whenever $m, n\geqslant k$.
\end{definition}

\begin{proposition}\label{thm6} Let $\{x_{n}\}$ a Cauchy sequence in $\star$-metric space $(X,d^\star)$. If $\{x_{n}\}$ has a accumulation point $x_{0}$, then $x_{n} \stackrel{d^\star}{\longrightarrow}x$.
\end{proposition}

\begin{proof}
For every $\epsilon> 0$, by Lemma \ref{lem3}, take an $\epsilon_{1}> 0$ such that $\epsilon_{1}\star \epsilon_{1}<\epsilon$. Since $\{x_{n}\}$ is Cauchy sequence, there exists $k_{1} \in \mathbb{N}$ such that $d^\star(x_{n},x_{m})< \epsilon_{1}$ whenever $m, n\geqslant k_{1}$. Noting that $x_{0}$ is the accumulation point of $\{x_{n}\}$, there exists $k_{2}\in \mathbb{N}$ such that $d^\star(x_{0},x_{n})< \epsilon_{1}$ whenever $n \geqslant k_{2}$. Therefore, choose $k=\max\{k_{1},k_{2}\}$, while $m \geqslant k$, we have $$d^\star(x_{0}, x_{m}) \leqslant d^\star(x_{0}, x_{n})\star d^\star(x_{n}, x_{m})< \epsilon_{1}\star \epsilon_{1}< \epsilon.$$ This shows that $\{x_{n}\}$ converges to $x_{0}$.
\end{proof}

\begin{definition}
A $\star$-metric space $(X,d^\star)$ is {\it complete} if every Cauchy sequence in $(X,d^\star)$ is convergent to a point of $X$.
\end{definition}

\begin{theorem}\label{thm13}
 A $\star$-metric space $(X,d^\star)$ is compact if and only if $(X,d^\star)$ is complete and totally bounded.
\end{theorem}

\begin{proof}
Necessity. Let $(X,d^\star)$ be a compact $\star$-metric space. According to Corollary \ref{col2}, $(X,d^\star)$ is totally bounded. According to Proposition \ref{thm6}, if a Cauchy sequence in space $(X,d^\star)$ has convergent subsequences, then this Cauchy sequence converges. Since compact $\star$-metric space is sequentially compact, which means that every sequence of points of $X$ has a convergent subsequence. By Corollary \ref{col1}, every Cauchy sequence in $(X,d^\star)$ is convergent to a point of $X$. This implies that $(X,d^\star)$ is complete.

Sufficiency. Let $(X,d^\star)$ be a complete and totally bounded $\star$-metric space. To conclude the proof it suffices to show that $X$ is sequentially compact which implies that $X$ is compact, by Corollary \ref{col1}.

Let $\{x_{n}\}$ be any sequence in $\star$-metric space $(X,d^\star)$. From the total boundedness of space $X$, there exists finite open-balls cover $X$ with radius $1$. At least one of the finite open-ball $B_{d^\star}^1$ contains infinite points $x_{n}$ in sequence $\{x_{n}\}$. Let the set formed by the subscript $n$ of $x_{n}$ contained in $B_{d^\star}^1$ be $N_{1}$, then $N_{1}$ is an infinite set, such that $x_{n}\in B_{d^\star}^1$ whenever $n \in N_{1}$. Then use finite open-balls to cover $X$ with radius $1/2$. Among these finite open-balls, there must be at least one open-ball $B_{d^\star}^2$ and an infinite subset $N_{2}$ of $N_{1}$, such that $x_{n}\in B_{d^\star}^2$ whenever $n \in N_{2}$. Generally speaking, taking the infinite subset $N_{k}$ of the positive integer set, we can select the open-ball $B_{d^\star}^{k+1}$ with radius of 1/(k+1) and the infinite set $N_{k+1}\subset N_{k}$, such that $x_{n}\in B_{d^\star}^{k+1}$ whenever $n \in N_{k+1}$.

Take $n_{1}\in N_{1}$, $n_{2}\in N_{2}$, which $n_{2}> n_{1}$. Generally speaking, $n_{k}$ has been taken, we can choose $n_{k+1}\in N_{k+1}$ so that $n_{k+1}> n_{k}$. Since for every $N_{k}$ is infinite set, the above method can be completed. For $i,j \geqslant k$, we have $n_{i},n_{j}\in N_{k}$ such that $x_{n_{i}},x_{n_{j}}\in B_{d^\star}^k$. This implies that $\{x_{n_{k}}\}$ is a Cauchy sequence, and it is convergent by completeness of $(X,d^\star)$.
\end{proof}

The {\it distance $D(x, A)$ from a point to a set $A$} in a $\star$-metric space $(X,d^\star)$ is defined by letting  $$D(A,x)=D(x, A)=\inf_{y \in A}\{d^\star(x,y)\}, ~if~ A\neq\emptyset, ~and~ D(x,\emptyset)=D(\emptyset, x)=1.$$

\begin{proposition}\label{thm8} Let $(X,d^\star)$ be a $\star$-metric space, and $A\subset X$. Then $\overline{A}=\{x:D(A,x)=0\}$.
\end{proposition}

\begin{proof}
For any $x_{0}\in \overline{A}$ there exists $\{x_{n}\}_{n \in \mathbb{N}}\subset A$ such that $x_{n} \stackrel{d^\star}{\longrightarrow} x_{0}$. This implies that $d^\star(x_{n},x_{0})\rightarrow 0$. Since $0\leqslant d^\star(x_{0},A)\leqslant d^\star(x_{n},x_{0})\rightarrow 0$, we have that $d^\star(x_{0},A)=0$, which implies that $x_{0}\in \{x:D(A,x)=0\}$. Therefore $\overline{A}\subseteq \{x:D(A,x)=0\}$.

Suppose the contrary that take $y\in \{x:D(A,x)=0\}$ which satisfies $d^\star(y,A)=0$ and $y\notin \overline{A}$. Then there exists $\epsilon_{0} >0$ such that $B_{d^\star}(y,\epsilon_{0})\cap A=\emptyset$, which implies that $d^\star(y,A)\geqslant \epsilon_{0}$. This is a contradiction with $d^\star(y,A)=0$. Thus $\overline{A}\supseteq \{x:D(A,x)=0\}$.

This shows that $\overline{A}= \{x:D(A,x)=0\}$.
\end{proof}

\begin{definition}
Let $A$ be a subset of $\star$-metric space $(X,d^\star)$. We define $\delta(A)=\sup_{x,y\in A}\{d^\star(x,y)\}$ as the {\it diameter} of the set $A$; it can be finite or equal to $\infty$. We also define $\delta(\emptyset)=0$.
\end{definition}

Cantor theorem is an important characterization of complete metric spaces. Similarly, we extend the Cantor theorem in metric spaces into $\star$-metric spaces.

\begin{theorem}\label{thm7}
A $\star$-metric space is complete if and only if for every decreasing sequence $F_{1}\supseteq F_{2}\supseteq F_{3}\supseteq\ldots$ of non-empty closed subsets of space $X$, such that $\lim_{n\rightarrow \infty }\delta(F_{n})=0$, the intersection $\bigcap_{n=1}^\infty F_{n}$ is a one-point set.
\end{theorem}

\begin{proof}
Necessity. Let $(X,d^\star)$ be a complete $\star$-metric space, and $F_{1}, F_{2},\dots$ a sequence of non-empty closed subsets of $X$ such that $$\lim_{n\rightarrow \infty }\delta(F_{n})=0 ~and~ F_{n+1}\subset F_{n}~ for~ n=1,2,\dots$$ Choose $x_{n}\in F_{n}$, for every $n\in \mathbb{N}$. Now we shall prove that $\{x_{n}\}$ is a Cauchy sequence. According to $\lim_{n\rightarrow \infty }\delta(F_{n})=0$, for $\epsilon >0$, there exists a $k \in \mathbb{N}$ such that $\delta(F_{n})< \epsilon$ where $ n> k$. When $n\geqslant m>k$, we have $x_{n}\in F_{n}\subset F_{m}$ since $\{F_{n}\}$ is a decreasing sequence. Furthermore $x_{m}\in F_{m}$, so that $$d^\star(x_{n},x_{m})< \delta(F_{m})<\epsilon .$$ So, $\{x_{n}\}$ is a Cauchy sequence and thus is convergent to a point $x_{0}\in X$. Thus, any neighborhood of $x_{0}$ intersects $F_{n}$ $(n=1,2,\dots)$. The sets $F_{n}$ being closed, we have $x_{0}\in \bigcap_{n=1}^\infty F_{n}$

Now, we need proof $\bigcap_{n=1}^\infty F_{n}=\{x_{0}\}$. If there is an $y\in \bigcap_{n=1}^\infty F_{n}$, by $lim_{n\rightarrow \infty }\delta(F_{n})=0$, choose a $k\in \mathbb{N}$ such that $\delta(F_{n})< \epsilon$ where $ n> k$. Then we have $$d^\star(x_{0},y)< \delta(F_{n})<\epsilon, ~where~ x_{0}, y\in F_{n}.$$ Arbitrariness by $\epsilon$, $d^\star(x_{0},y)=0$, so $x_{0}=y$.

Sufficiency. Let $\{x_{n}\}$ be a Cauchy sequence of $(X,d^\star)$. For every $k\in \mathbb{N}$, there exists $l_{k}, ~r_{k}$ such that $d^\star (x_{l_{k}},x_{n})\leqslant \frac{1}{r_{k+1}}$ where $ n> l_{k}$. Let $x_{l_{k}}$ be the smallest positive integer with the above properties, so that $l_{k}\leqslant l_{k+1}, r_{k}\leqslant r_{k+1}(k=1,2,\dots)$. Construct the following sequence of closed subset $\{F_{n}\}$ which defined by letting $$F_{k}=\overline{B_{d^\star}(x_{l_{k}},\frac{1}{r_{k}})},~ (k=1,2,\dots),$$ here $B_{d^\star}(x_{l_{k}},\frac{1}{r_{k}})=\{y \in X:d^\star(x_{l_{k}},y)\leqslant \frac{1}{r_{k}}\}$. By Proposition \ref{thm8}, we can get that $d^\star (F_{k})\leqslant \frac{2}{r_{k}}$, this implies that $\lim_{n\rightarrow \infty }\delta(F_{n})=0$.

Now take subset $\{H_{n}\}$ of $\{F_{n}\}$. We define $H_{1}=F_{k_{1}}, k_{1}=1$. Then take $H_{2}=F_{k_{2}}$, by Lemma \ref{lem3}, we can set that $k_{2}=\min \{j\geqslant 2: \frac{1}{r_{j}}\star \frac{1}{r_{j}}<\frac{1}{r_{k_{1}}}\}$. Generally speaking, if we take the positive integer $k_{n}$, we can take $k_{n+1}=\min \{j\geqslant k_{n}+1: \frac{1}{r_{j}}\star \frac{1}{r_{j}}<\frac{1}{r_{k_{n}}}\}$, such that $H_{n+1}=F_{k_{n+1}}$. Now we shall show that $\{H_{n}\}$ satisfies the conditions in our theorem.

Let $y\in H_{n+1}$, then according to the selection method of $l_{k}$, we have $$d^\star(y,x_{l_{k_{n+1}}})< \frac{1}{r_{k_{n+1}}}, ~d^\star(x_{l_{k_{n+1}}},x_{l_{k_{n}}})< \frac{1}{r_{k_{n}+1}},$$
thus $$d^\star(y,x_{l_{k_{n}}})\leqslant d^\star(y,x_{l_{k_{n+1}}})\star d^\star(x_{l_{k_{n+1}}},x_{l_{k_{n}}})< \frac{1}{r_{k_{n+1}}}\star \frac{1}{r_{k_{n}+1}}< \frac{1}{r_{k_{n+1}}} \star \frac{1}{r_{k_{n+1}}}< \frac{1}{r_{k_{n}}}.$$ This implies that $y\in H_{n}$, i.e. $H_{n+1}\subset H_{n}$.

According to the assumption, there should be  $\bigcap_{n=1}^\infty F_{n}=\{x_{0}\}$. Now we shall show that a Cauchy sequence $\{x_{n}\}$ is convergent to a point $x_{0}$. For every $\epsilon> 0$, by Lemma \ref{lem3}, take a $r_{k}\in \mathbb{N}$ such that $\frac{1}{r_{k}}\star \frac{1}{r_{k}}<\epsilon$, and $d^\star(x_{l_{k}},x_{n})\leqslant \frac{1}{r_{k+1}}$. In addition, $x_{0}\in F_{k}$, $d^\star(x_{l_{k}},x_{0})\leqslant \frac{1}{r_{k}}$, then we have $$d^\star(x_{0},x_{n})\leqslant d^\star(x_{0},x_{l_{k}})\star d^\star(x_{l_{k}},x_{n})< \frac{1}{r_{k}}\star \frac{1}{r_{k+1}}< \frac{1}{r_{k}}\star \frac{1}{r_{k}}< \epsilon.$$ Thus, $\lim_{n\rightarrow \infty }x_{n}=x_{0}$, $(X,d^\star)$ is a complete $\star$-metric space.
\end{proof}

\begin{theorem}
A $\star$-metric space is complete if and only if for every family of closed subsets of $X$ which has the finite intersection property and for every $\epsilon>0$ contains a set of diameter less than $\epsilon$ has non-empty intersection.
\end{theorem}

\begin{proof}
Sufficiency of the condition in our theorem for completeness of a $\star$-metric space follows from the Theorem \ref{thm7}.

We shall show that the condition holds in every complete $\star$-metric space $(X,d^\star)$. Consider a family $\{F_{s}\}_{s\in S}$ of closed subsets of $X$ which has the finite intersection property and which for every $j\in \mathbb{N}$ contains a set $F_{s_{j}}$, such that $\delta(F_{s_{j}})< \frac{1}{j}$. Let $F_{i}=\bigcap_{j\leqslant i}F_{s_{j}}$.
One easily sees that the sequence $F_{1}, F_{2},\dots$ satisfies the condition of the Cantor theorem, since $F_{n+1}\subset F_{n}$ and  $\delta(F_{i})\leqslant \delta(F_{s_{i}})< \frac{1}{i}$ which means $\lim_{n\rightarrow \infty }\delta(F_{n})=0$. 
So that there exists an $x\in \bigcap_{i=1}^\infty F_{i}$. Clearly, we have $\bigcap_{i=1}^\infty F_{i}=\{x\}$. Now, let us take an arbitrary $s_{0}\in S$; letting $F_{i}'=F_{s_{0}}\cap F_{i}$ for $i=1,2,\dots$ we obtain again a sequence $F_{1}',F_{2}',\dots$ satisfying the conditions of the Theorem \ref{thm7}. Since $$\emptyset \neq \bigcap_{i=1}^\infty F_{i}'=F_{s_{0}}\cap \bigcap_{i=1}^\infty F_{i}=F_{s_{0}}\cap \{x\},$$ we have $x\in F_{s_{0}}$. Hence $x\in \bigcap_{s\in S}F_{s}$.
\end{proof}

\begin{theorem} \label{thm10}
A subspace $(M,d^\star)$ of a complete $\star$-metric space $(X,d^\star)$ is complete if and only if $M$ is closed in $X$.
\end{theorem}

\begin{proof}
Necessity. Let $x\in \overline{M}$, and we define $F_{k}=M\cap \overline{B_{d^\star}(x,\frac{1}{k})}(k=1,2,\dots)$, then sequence $\{F_{k}\}$ is non-empty closed subsets in subspace $M$, so one can easily check that $\{F_{k}\}$ satisfies the conditions $(1),(2)$ in the Theorem \ref{thm7}. Since subspace $(M,d^\star)$ is complete, by Theorem \ref{thm7}, obviously $\bigcap_{k=1}^\infty F_{k}=\{x\}$, it follows that $x\in M$. Therefore $M= \overline{M}$.

Sufficiency. Let $M$ is a closed set, every Cauchy sequence of $\star$-metric space $(M,d^\star)$ is also a Cauchy sequence of complete $\star$-metric space $(X,d^\star)$, so it converges to a certain point $x\in X$. Since $M$ is closed in $X$, $x\in M$. This completes the proof.
\end{proof}

The following theorem shows that in a class of $\star$-metric spaces, the completeness is preserved by finite products.

\begin{theorem}\label{thm11} Let $\{(X_{i},d_{i}^\star)\}_{i=1}^n $ be a family of finite nonempty $\star$-metric spaces and $X=\prod_{i=1}^{n}X_{i}$ the Cartesian product. Then
\begin{enumerate}
\item[(1)]$X$ with the $\star$-metric $d_{T}^\star$ defined by formula (3.1) is complete if and only if all $\star$-metric spaces $(X_{i},d_{i}^\star)$ are complete;
\item[(2)]$X$ with the $\star$-metric $d_{max}^\star$ defined by formula (3.2) is complete if and only if all $\star$-metric spaces $(X_{i},d_{i}^\star)$ are complete.
\end{enumerate}
\end{theorem}

\begin{proof}
(1). Assume that the space $(X,d_{T}^\star)$ is complete. For every subspace $X_{m}^*=\prod_{i=1}^{n}A_{i}$ of $X$, where $A_{m}=X_{m}$ and $A_{i}=\{x_{i}^*\}$ is a one-point subset of $X_{i}$ for $i \neq m$, is closed in $(X,d_{T}^\star)$. Then the subspace $X_{m}^*$ is complete by Theorem \ref{thm10}. One can easily verify that $p_{m}^*=p_{m}\mid _{X_{m}^*}:X_{m}^*\rightarrow X_{m}$ is a isometric isomorphism, since $d_{T}^\star|_{X_{m}^*}(p_{m}^*(x),p_{m}^*(y))= d_{2}^\star(x,y)$. Therefore, for every Cauchy sequence $\{x_{n}\}$ in $(X_{m},d_{m}^\star)$, the sequence $\{{p_{m}^*}^{-1}(x_{n})\}$ is a Cauchy sequence in $X_{m}^*$. Then $$p_{m}^*(\lim_{n\rightarrow \infty} {p_{m}^*}^{-1}(x_{n}))=\lim_{n\rightarrow \infty}x_{n},$$ so that the space $(X_{m},d_{m}^\star)$ is complete.

 Assume  that all spaces $(X_{i},d_{i}^\star)$ are complete. Take any Cauchy sequence $\{y_{k}\}_{k \in \mathbb{N}}$ in $(X,d_{T}^\star)$, where $y_{k}=(x_{i}^k)$, for $1 \leqslant i\leqslant n$. Then the sequence $\{x_{i}^k\}_{k \in \mathbb{N}}$ is a Cauchy sequence in $(X_{i},d_{i}^\star)$ and thus converges to a point $x_{i}^0 \in X_{i}$. Now, we shall show that $\{y_{k}\}_{k \in \mathbb{N}}$ converges to a point $x^0 = (x_{i}^0)$. For $\epsilon> 0$, by Lemma \ref{lem3}, take an $\epsilon_{1}> 0$ such that $ \overbrace{\epsilon_{1}\star \epsilon_{1}\star \dots \star \epsilon_{1}}^{n \text{ times}}<\epsilon$. Since $\{x_{i}^k\}_{k \in \mathbb{N}}$ converges to a point $x_{i}^0$, there exists $m_{i} \in \mathbb{N}$, such that $d_{i}^\star(x_{i}^k,x_{i}^0)< \epsilon_{1}$, where $k\geqslant m_{i}$. Thus choose $m=\max_{1 \leqslant i\leqslant n}\{m_{i}\}$, such that $$d_{T}^\star(y_{k},x^0)=d_{1}^\star(x_{1}^k,x_{1}^0)\star d_{2}^\star(x_{2}^k,x_{2}^0)\star \dots \star d_{n}^\star(x_{n}^k,x_{n}^0)< \overbrace{\epsilon_{1}\star \epsilon_{1}\star \dots \star \epsilon_{1}}^{n \text{ times}}<\epsilon,$$ whenever $k\geqslant m$. We have shown that $(X,d_{T}^\star)$ is complete.

 (2). Assume that the space $(X,d_{T}^\star)$ is complete. The method of proof is the same as (1).

Assume that all spaces $(X_{i},d_{i}^\star)$ are complete. Take any Cauchy sequence $\{y_{k}\}_{k \in \mathbb{N}}$ in $(X,d_{\max}^\star)$, where $y_{k}=(x_{i}^k)$, for $1 \leqslant i\leqslant n$. Then the sequence $\{x_{i}^k\}_{k \in \mathbb{N}}$ is a Cauchy sequence in $(X_{i},d_{i}^\star)$ and thus converges to a point $x_{i}^0 \in X_{i}$. Now, we shall show that $\{y_{k}\}_{k \in \mathbb{N}}$ converges to a point $x^0 = (x_{i}^0)$. Since $\{x_{i}^k\}_{k \in \mathbb{N}}$ converges to a point $x_{i}^0$. For every $\epsilon > 0$ there exists $m_{i} \in \mathbb{N}$, such that $d_{i}^\star(x_{i}^k,x_{i}^0)< \epsilon$, where $k\geqslant m_{i}$. Without loss of generality, let $\max_{1\leqslant i\leqslant n} d_{i}^\star(x_{i}^k,x_{i}^0)=d_{j}^\star(x_{j}^k,x_{j}^0)$, then while $k\geqslant m=m_{j}$, such that $$d_{\max}^\star(y_{k},x^0)=\max_{1\leqslant i\leqslant n}d_{i}^\star(x_{i}^k,x_{i}^0)=d_{j}^\star(x_{j}^k,x_{j}^0)< \epsilon.$$  We have shown that $(X,d_{\max}^\star)$ is complete.
\end{proof}







\begin{theorem} \label{thm11.2}
If $\{(X_{\alpha},d_{\alpha}^\star)\}_{\alpha\in A}$ be a family of $\star$-metric spaces such that the metric $d_{i}^\star$ is bounded for each $\alpha\in A$, and $X=\bigoplus_{\alpha\in A}X_{\alpha}$  be the disjoint union of $\{X_{\alpha}\}$. Then $(X,d^\star_{q})$ defined by formula (3.3) is complete if and only if all spaces $(X_{\alpha},d_{\alpha}^\star)$ are complete.
\end{theorem}

\begin{proof}
Necessity. Assume that $(X,d^\star_{q})$ is complete. Then it is easy to see that all sets $X_{\alpha}$ are open-and-closed in $X$. So, all spaces $(X_{\alpha},d_{\alpha}^\star)$ are complete by Theorem \ref{thm10}.

Sufficiency. Assume that all spaces $(X_{\alpha},d_{\alpha}^\star)$ are complete. Then we have $(X,d^\star_{q})$ is complete, because every Cauchy sequence of $\star$-metric space $(X_{\alpha},d_{\alpha}^\star)$ is also a Cauchy sequence of $(X,d^\star_{q})$ and it converges to a certain point $x\in X_{\alpha}\subseteq X$.
\end{proof}

Baire theorem is a very important result in complete metric spaces. We shall extend this theorem to  complete
$\star$-metric spaces.

\begin{theorem}\label{thm12}
In a complete $\star$-metric space $(X,d^\star)$ the intersection $A=\bigcap_{n=1}^\infty A_{n}$ of a sequence $A_{1},A_{2},\dots$ of dense open subsets is a dense set.
\end{theorem}

\begin{proof}
Let $A=\bigcap_{n=1}^\infty A_{n}$, for every $A_{n}$ is an open dense subset of complete $\star$-metric space $(X,d^\star)$. Now, construct the  sequence of closed subset $\{F_{n}\}$ which satisfies conditions in Theorem \ref{thm7}. Since $A_{1}$ is dense in $X$, and $U$ is a non-empty open set, then $A_{1}\cap U\neq \emptyset$. Take $x_{1}\in A_{1}\cap U$, since $ A_{1}\cap U$ is an open set, there exists $\epsilon _{1}$ which satisfies $0< \epsilon _{1}< 1/2^2$, such that $\overline{B_{d^\star}(x_{1},\epsilon _{1})} \subset A_{1}\cap U$. Since $A_{2}$ is dense in $X$, and $B_{d^\star}(x_{1},\epsilon _{1})$ is an open set, then $A_{2}\cap B_{d^\star}(x_{1},\epsilon _{1})\neq \emptyset$. Take $x_{2}\in A_{2}\cap B_{d^\star}(x_{1},\epsilon _{1})$, since $ A_{2}\cap B_{d^\star}(x_{1},\epsilon _{1})$ is an open set, there exists $\epsilon _{2}$ which satisfies $0< \epsilon _{2}< \epsilon _{1}/2$, such that $\overline{B_{d^\star}(x_{2},\epsilon _{2})} \subset A_{2}\cap B_{d^\star}(x_{1},\epsilon _{1})$. Obviously, $\overline{B_{d^\star}(x_{2},\epsilon _{2})} \subset \overline{B_{d^\star}(x_{1},\epsilon _{1})}$ and $\overline{B_{d^\star}(x_{2},\epsilon _{2})} \subset A_{2}\cap U$. Going on, one can easily obtain the sequence of closed subset $\{F_{n}\}=\{\overline{B_{d^\star}(x_{n},\epsilon _{n})}\}$ which satisfies $F_{n+1}\subset F_{n}$ and $\delta (F_{n}) \leqslant 1/2^n$ (n=1,2,\dots). This implies that $\{F_{n}\}$ satisfies conditions in Theorem \ref{thm7}. Noting that $F_{n}\subset A_{n}\cap U$, by Theorem \ref{thm7},  $\bigcap_{n=1}^\infty F_{n}\neq \emptyset$, then we have $$A\cap U=(\bigcap_{n=1}^\infty A_{n}) \cap U=\bigcap_{n=1}^\infty (A_{n} \cap U)\supset \bigcap_{n=1}^\infty F_{n} \neq \emptyset,$$ this implies that $A$ is dense in $X$.
\end{proof}

Every metric space is isometric to a subspace of a complete metric space. It would be interesting to find out whether this result remain valid in the class of $\star$-metric spaces:

\begin{question}
Is every $\star$-metric space isometric to a subspace of a complete $\star$-metric space?
\end{question}












\end{document}